\documentclass[runningheads]{llncs}
\usepackage{graphicx}
\usepackage{amsmath,amssymb}
\usepackage{tikz}
\usetikzlibrary{arrows, automata, shapes, positioning}
\usetikzlibrary{arrows.meta}
\usetikzlibrary{decorations.pathmorphing}
\usepackage{caption}
\usepackage{subcaption}
\usepackage{hyperref}
\usepackage{todonotes}
\newcommand{\N}{\mathbb{N}}
\newcommand{\NU}{\mathcal{N}_U}

\newcommand{\rep}{\mathrm{rep}}

\newcommand{\Pref}{\mathrm{Pref}}
\newcommand{\Suff}{\mathrm{Suff}}
\newcommand{\Fac}{\mathrm{Fac}}
\newcommand{\floor}[1]{\left\lfloor#1\right\rfloor}
\newcommand{\ceil}[1]{\left\lceil#1\right\rceil}
\newcommand{\lex}{\mathrm{lex}}

\begin{document}
\title{A full characterization of Bertrand numeration systems\thanks{\'Emilie Charlier is supported by the FNRS grant J.0034.22. \\ 
C\'elia Cisternino is supported by the FNRS grant 1.A.564.19F. \\
Manon Stipulanti is supported by the FNRS grant 1.B.397.20.}}
%
%
\author{\'Emilie Charlier\orcidID{0000-0002-5789-2674} \and C\'elia Cisternino\orcidID{0000-0002-1994-9625} \and
Manon Stipulanti\orcidID{0000-0002-2805-2465}}
\authorrunning{\'E. Charlier et al.}
%
\institute{Department of Mathematics, University of Li\`ege, Li\`ege, Belgium
\email{\{echarlier,ccisternino,m.stipulanti\}@uliege.be}}
\maketitle              
\begin{abstract}

Among all positional numeration systems, the widely studied Bertrand numeration systems are defined by a simple criterion in terms of their numeration languages. In 1989, Bertrand-Mathis characterized them via representations in a real base $\beta$. However, the given condition turns to be not necessary. Hence, the goal of this paper is to provide a correction of Bertrand-Mathis' result. The main difference arises when $\beta$ is a Parry number, in which case are derived two associated Bertrand numeration systems.
Along the way, we define a non-canonical $\beta$-shift and study its properties analogously to those of the usual canonical one.

\end{abstract}
%
%
%

\section{Introduction}

In $1957$, R\'enyi~\cite{Renyi:1957} introduced representations of real numbers in a real base $\beta>1$. A \emph{$\beta$-representation} of a nonnegative real number $x$ is an infinite sequence $a_1a_2\cdots$ over $\N$ such that $x=\sum_{i=1}^\infty \frac{a_i}{\beta^{i}}$.
The most commonly used algorithm in order to obtain such digits $a_i$ is the greedy algorithm.
The corresponding distinguished $\beta$-representation of a given $x\in[0,1]$ is called the $\beta$-\emph{expansion} of $x$ and is obtained as follows: set $r_0(x)=x$ and for all $i\ge 1$, let $\varepsilon_i(x)=\floor{\beta\, r_{i-1}(x)}$ and $r_i(x)=\beta\, r_{i-1}(x) - \varepsilon_i(x)$. 
The $\beta$-expansion of $x$ is the infinite word $d_{\beta}(x)=\varepsilon_1(x)\varepsilon_2(x)\cdots$ written over the alphabet $\{0,\ldots, \floor{\beta}\}$.  
In this theory, the $\beta$-expansion of $1$ and the \emph{quasi-greedy $\beta$-expansion} of 1 given by
$d_\beta^*(1)=\lim_{x\to 1^-}d_\beta(x)$ play crucial roles, as well as the \emph{$\beta$-shift}
\[
	S_\beta=\{w\in \{0,\ldots,\ceil{\beta}-1\}^\N : \forall i\ge 0,\ \sigma^i(w)\le_{\lex} d_{\beta}^*(1)\}
\]
where $\sigma(w_1w_2\cdots)$ denotes the shifted word $w_2w_3\cdots$. Parry~\cite{Parry1960} showed that   the $\beta$-shift $S_\beta$ is the topological closure (w.r.t. the prefix distance) of the set of infinite words that are the $\beta$-expansions of some real number in $[0,1)$ and Bertrand-Mathis~\cite{Bertrand-Mathis1986} characterized the real bases $\beta$ for which $S_\beta$ is sofic, i.e., its factors form a language that is accepted by a finite automaton. Expansions in a real base are extensively studied under various points of view and we can only cite a few of the many possible references~\cite{Bertrand-Mathis1986,Dajani&Kraaikamp2002,Lothaire2002,Parry1960,Schmidt1980}.

In parallel, other numeration systems are also widely studied, this time to represent nonnegative integers. A {\em positional numeration system} is given by an increasing integer sequence $U=(U(i))_{i\ge 0}$ such that $U(0)=1$ and the quotients $\frac{U(i+1)}{U(i)}$ are bounded. The \emph{greedy $U$-representation} of $n\in\N$, denoted $\rep_U(n)$, is the unique word
$a_1 \cdots a_\ell$ over $\N$ such that $n=\sum_{i=1}^\ell a_i U(\ell-i)$, $a_1 \neq 0$ and for all
$j \in \{1,\ldots,\ell\}$, $\sum_{i=j}^\ell a_iU(\ell-i)< U(\ell-j+1)$. These representations are written over the finite alphabet $A_U=\{0,\ldots,\sup_{i\ge 0} \big\lceil \frac{U(i+1)}{U(i)}\big\rceil - 1\}$. The \emph{numeration language} is the set $\NU=0^*\rep_U(\N)$.
Similarly, the literature about positional numeration systems is vast; see~\cite{Bertrand-Mathis1989,BruyereHansel1997,CharlierRampersadRigoWaxweiler2011,Loraud1995,MassuirPeltomakiRigo2019,Point2000,Shallit1994} for the most topic-related ones.

There exists an intimate link between $\beta$-expansions and greedy $U$-representa\-tions. Its study goes back to the work~\cite{Bertrand-Mathis1989} of Bertrand-Mathis. A positional numeration system $U$ is called \emph{Bertrand} if the corresponding numeration language $\NU$ is both prefix-closed and \emph{prolongable}, i.e., if for all words $w$ in $\NU$, the word $w0$ also belongs to $\NU$. These two conditions can be summarized as
\begin{align}
\label{Eq : BertrandCondition}
 \forall w\in A_U^*,\ w\in\NU \iff w0 \in\NU.
\end{align} 
The usual integer base numeration systems are Bertrand, as well the Zeckendorf numeration system~\cite{Zeckendorf1972}. 
This form of the definition of Bertrand numeration systems, as well as their names after Bertrand-Mathis, was first given in~\cite{BruyereHansel1997}, and then used in~\cite{CharlierRampersadRigoWaxweiler2011,MassuirPeltomakiRigo2019,Point2000}. Bertrand numeration systems were also reconsidered in~\cite{Loraud1995}. 
Moreover, the normalization in base $\beta>1$ in~\cite{BruyereHansel1997,FrougnySolomyak1996}  deals with these Bertrand numeration systems. 

In~\cite{Bertrand-Mathis1989}, Bertrand-Mathis showed that a positional numeration system $U$ is Bertrand if and only if there exists a real number $\beta>1$ such that $\NU=\Fac(S_\beta)$. In this case, $A_U=\{0,\ldots,\ceil{\beta}-1\}$ and for all $i\ge 0$, 
\begin{equation}
\label{Eq : Bertrand}
	U(i)=d_1 U(i-1)+d_2 U(i-2)+ \cdots + d_iU(0)+1
\end{equation}
where $(d_i)_{i\ge 1}=d^*_\beta(1)$. This result has been widely used, see for example~\cite{BruyereHansel1997,CharlierRampersadRigoWaxweiler2011,Lothaire2002}. However, the condition stated above is \emph{not necessary} (see Section~\ref{Sec : ProofThm}). Note that it is trivially sufficient. Therefore, in this work, we propose a correction of this famous theorem by fully characterizing Bertrand numeration systems.

The paper is organized as follows. We first fix some notation in Section~\ref{Sec : Preli}. In Section~\ref{Sec : ProofThm}, we illustrate the fact that the Bertrand-Mathis theorem stated above does not fully characterize Bertrand numeration systems and we provide a correction of this result. Then, in Section~\ref{Sec : Linear}, we investigate Bertand numeration systems based on a sequence that satisfies a linear recurrence relation. In Section~\ref{Sec : LexMax}, we obtain a second characterization of Bertrand numeration systems in terms of the lexicographically greatest words of each length in $\mathcal{N}_U$. This provides a refinement of a result of Hollander~\cite{Hollander1998}. Finally, seeing the importance of the newly introduced non-canonical $\beta$-shift, we study its main properties in Section~\ref{Sec : NonCanonicalShift}.

\section{Basic notation}
\label{Sec : Preli}

We make use of common notions in formal language theory, such as alphabet, letter, word, length of a word, prefix distance, convergence of words, language, code and automaton~\cite{Lothaire2002}. In particular, the length of a finite word $w$ is denoted by $|w|$. The notation $w^\omega$ means an infinite repetition of the finite word $w$. The set of factors of a word $w$ is written $\Fac(w)$ and the set of factors of words in a set $L$ is written $\Fac(L)$. Given a finite word $w$ and $n\in \{1,\ldots, |w|\}$, the \emph{prefix} and \emph{suffix} of length $n$ of $w$ are respectively written $\Pref_n(w)$ and $\Suff_n(w)$. Similarly, for an infinite word $w$ and $n\ge 0$, we let $\Pref_n(w)$ denote the prefix of length $n$ of $w$. 

Let $(A,<)$ be a totally ordered alphabet. The order $<$ on $A$ induces the following orders on words over $A$. For two length-$n$ words $u,v\in A^*$, we write $u<_{\lex}v$ if there exists $\ell\in \{1,\ldots,n\}$ such that $\Pref_{\ell-1}(u)<_{\lex}\Pref_{\ell-1}(v)$ and $u_\ell<v_{\ell}$, and we write $u\le_{\lex}v$ if either $u<_{\lex}v$ or $u=v$. For two infinite words $u,v \in A^\N$, we write $u<_{\lex}v$ if there exists $n\ge1$ such that $\Pref_n(u)<_{\lex}\Pref_n(v)$. In both cases, if $u<_{\lex}v$ then we say that $u$ is \emph{lexicographically less} than $v$.

\section{Characterization of Bertrand numeration systems}
\label{Sec : ProofThm}

The goal of this section is to give a full characterization of Bertrand numeration systems defined by~\eqref{Eq : BertrandCondition}. In doing so, we correct the result of Bertand-Mathis stated in the introduction.

First, we note that both implications in~\eqref{Eq : BertrandCondition} are relevant. This observation is illustrated in the following example.

\begin{example}
Consider the numeration system $U$ defined by $(U(0),U(1))=(1,3)$ and $U(i)=U(i-1)+U(i-2)$ for all $i\ge 2$. It is not Bertrand as its numeration language is not prolongable: for instance, $2\in\NU$ but $20\notin\NU$.

Now, consider $U$ defined by $(U(0), U(1))= (1,2)$ and  $U(i)= 5U(i-1)+U(i-2)$ for all $i\ge 2$. It is not Bertrand since the corresponding language $\NU$ is not prefix-closed. Indeed, $50 \in \rep_U(\N)$ but $5 \notin \rep_U(\N)$. 
\end{example}

Then, let us show that the condition given in the original Bertrand-Mathis result characterizing the Bertrand numeration systems is \emph{not necessary}. This fact was already observed in~\cite{MassuirPeltomakiRigo2019}.

\begin{example}
\label{Ex : CounterExampleBertrand} \
Consider the positional numeration system  $U$ defined by $U(0)=1$ and $U(i)=3U(i-1)+1$ for all $i \ge 1$. This is the example given in~\cite{MassuirPeltomakiRigo2019}. It is easy to see that $\NU=\{0,1,2\}^* \cup \{0,1,2\}^* 3 0^*$. The minimal automaton of this language is depicted in Figure~\ref{Fig : SubNonCanonical3}. Therefore, $U$ is Bertrand. However, $U$ satisfies~\eqref{Eq : Bertrand} with $(d_n)_{n\ge 1}$ not equal to $d^*_3(1)=2^\omega$ as prescribed by the result from~\cite{Bertrand-Mathis1989} (which has been transcribed in the introduction) but equal to $d_3(1)=30^\omega$ instead. 

Another example is the following one. We consider the positional numeration system $U$ defined by $(U(0),U(1)) =(1,2)$ and $U(i)=U(i-1)+U(i-2)+1$ for all $i\ge2$. This system is Bertrand since the corresponding numeration language is $\NU=\{0,10\}^*\cup \{0,10\}^*1\cup \{0,10\}^*110^*$. The minimal automaton of this language is depicted in Figure~\ref{Fig : SubNonCanonicalPhi}. The sequence $U$ satisfies~\eqref{Eq : Bertrand} with $(d_n)_{n\ge 1}$ equal to $d_\varphi(1)=110^\omega$. 

\begin{figure}[tb]
\begin{subfigure}{0.45\textwidth}
\begin{center}
\begin{tikzpicture}
\tikzstyle{every node}=[shape=circle, fill=none, draw=black,
minimum size=20pt, inner sep=2pt]
\node(1) at (0,0) {$ $};
\tikzstyle{every node}=[shape=circle, fill=none, draw=black,
minimum size=15pt, inner sep=2pt]
\node(Af) at (0,0) {} ;
\tikzstyle{every path}=[color=black, line width=0.5 pt]
\tikzstyle{every node}=[shape=circle, minimum size=5pt, inner sep=2pt]
\draw [-Latex] (-1,0) to node {} (1);  
\draw [-Latex] (1) to [loop above] node [above=-0.2] {$0,1,2$} (1) ;
\end{tikzpicture}
\caption{$U(i)=3^i$ for $i\ge 0$.}
\label{Fig : SubCanonical3}
\end{center}
\end{subfigure}
\begin{subfigure}{0.5\textwidth}
\begin{center}
\begin{tikzpicture}
\tikzstyle{every node}=[shape=circle, fill=none, draw=black,
minimum size=20pt, inner sep=2pt]
\node(1) at (0,0) {$ $};
\node(2) at (2,0) {$ $} ;
\tikzstyle{every node}=[shape=circle, fill=none, draw=black,
minimum size=15pt, inner sep=2pt]
\node(Af) at (0,0) {} ;
\node(Bf) at (2,0) {} ;
\tikzstyle{every path}=[color=black, line width=0.5 pt]
\tikzstyle{every node}=[shape=circle, minimum size=5pt, inner sep=2pt]
\draw [-Latex] (-1,0) to node {} (1);  
\draw [-Latex] (1) to [loop above] node [above=-0.2] {$0,1,2$} (1) ;
\draw [-Latex] (2) to [loop above] node [above] {$0$} (2) ;
\draw [-Latex] (1) to node [above]
{$3$} (2) ;
\end{tikzpicture}
\caption{$U(0)=1$ and $U(i)=3U(i-1)+1$ for $i \ge 1$.}
\label{Fig : SubNonCanonical3}
\end{center}
\end{subfigure}
\caption{The minimal automata of the languages $\NU$ where $U$ are respectively the canonical and non-canonical Bertrand numeration systems associated with $3$.}
\label{Fig : CounterExBertrand}
\end{figure}
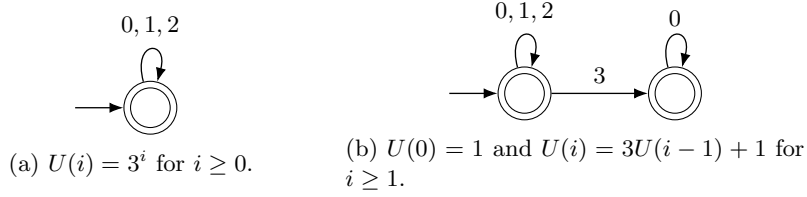

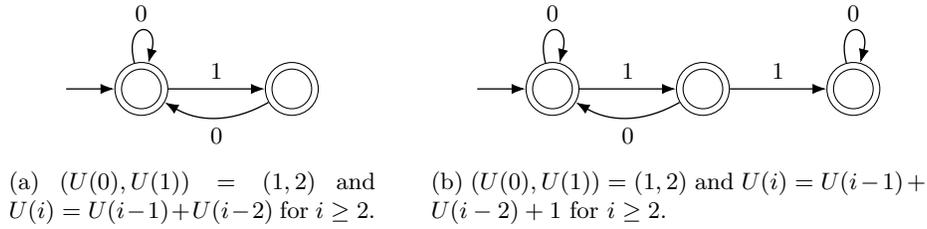
\begin{figure}[tb]
\begin{subfigure}{0.4\textwidth}
\begin{center}
\begin{tikzpicture}
\tikzstyle{every node}=[shape=circle, fill=none, draw=black,minimum size=20pt, inner sep=2pt]
\node(1) at (0,0) {$ $};
\node(2) at (2,0) {$ $} ;
\tikzstyle{every node}=[shape=circle, fill=none, draw=black,minimum size=15pt, inner sep=2pt]
\node(Af) at (0,0) {} ;
\node(Bf) at (2,0) {} ;
\tikzstyle{every path}=[color=black, line width=0.5 pt]
\tikzstyle{every node}=[shape=circle, minimum size=5pt, inner sep=2pt]
\draw [-Latex] (-1,0) to node {} (1);  
\draw [-Latex] (1) to [loop above] node [above] {$0$} (1) ;
\draw [-Latex] (2) to [bend left] node [below] {$0$} (1) ;
\draw [-Latex] (1) to node [above]
{$1$} (2) ;
\end{tikzpicture}
\caption{$(U(0),U(1))=(1,2)$ and $U(i)=U(i-1)+U(i-2)$ for $i\ge2$.}
\label{Fig : SubCanonicalPhi}
\end{center}
\end{subfigure}
\hfill
\begin{subfigure}{0.54\textwidth}
\begin{center}
\begin{tikzpicture}
\tikzstyle{every node}=[shape=circle, fill=none, draw=black,minimum size=20pt, inner sep=2pt]
\node(1) at (0,0) {$ $};
\node(2) at (2,0) {$ $} ;
\node(3) at (4,0) {$ $} ;
\tikzstyle{every node}=[shape=circle, fill=none, draw=black,minimum size=15pt, inner sep=2pt]
\node(Af) at (0,0) {} ;
\node(Bf) at (2,0) {} ;
\node(Cf) at (4,0) {} ;
\tikzstyle{every path}=[color=black, line width=0.5 pt]
\tikzstyle{every node}=[shape=circle, minimum size=5pt, inner sep=2pt]
\draw [-Latex] (-1,0) to node {} (1);  
\draw [-Latex] (1) to [loop above] node [above] {$0$} (1) ;
\draw [-Latex] (2) to [bend left] node [below] {$0$} (1) ;
\draw [-Latex] (1) to node [above]
{$1$} (2) ;
\draw [-Latex] (2) to node [above]{$1$} (3) ;
\draw [-Latex] (3) to [loop above] node [above] {$0$} (3) ;
\end{tikzpicture}
\caption{$(U(0),U(1))=(1,2)$ and $U(i)=U(i-1)+U(i-2)+1$ for $i\ge2$.}
\label{Fig : SubNonCanonicalPhi}
\end{center}
\end{subfigure}
\caption{The minimal automata of the languages $\NU$ where $U$ are respectively the canonical and non-canonical Bertrand numeration systems associated with~$\frac{1+\sqrt{5}}{2}$.}
\label{Fig : NC-Zeckendorf}
\end{figure}
\end{example}
 
We will show that, up to a single exception, the only possible Bertrand numeration systems are given by the recurrence relation~\eqref{Eq : Bertrand} where the sequence of coefficients $(d_n)_{n\ge 1}$ is either equal to $d^*_\beta(1)$ or to $d_\beta(1)$, as is the case of the previous two systems. Before proving our characterization of Bertrand numeration systems, we need some technical results.

\begin{lemma}
\label{Lem : RepUEqualSuff}
The numeration language $\NU$ of a positional numeration system $U$ is equal to $\{a\in A_U^* : \forall i\le |a|,\ \Suff_i(a) \le_{\lex} \rep_U(U(i)-1)\}$.
\end{lemma}

\begin{proof}
The result follows from the fact that $\rep_U(U(i)-1)$ is the lexicographically greatest word of length $i$ in $\NU$.
\end{proof}

\begin{lemma}
\label{Lem : FacBertrand}
The numeration language $\NU$ of a Bertrand numeration system $U$ is factorial, that is,\ $\Fac(\NU)=\NU$.
\end{lemma}

\begin{proof} 
The fact that $\NU$ is prefix-closed comes from the definition of a Bertrand numeration system. Since any positional numeration system has a suffix-closed numeration language, the conclusion follows.
\end{proof}

\begin{lemma}
\label{Lem : ConvergenceBertrand}
A positional numeration system $U$ is Bertrand if and only if there exists an infinite word $a$ over $A_U$ such that $\rep_U(U(i)-1)=\Pref_i(a)$ for all $i\ge 0$. In this case, we have $\sigma^i(a)\le_{\lex} a$ for all $i\ge 1$.
\end{lemma}

\begin{proof}
In order to get the necessary condition, it suffices to show that if $U$ is a Bertrand numeration system then for all $i\ge 1$, $\rep_U(U(i)-1)$ is a prefix of $\rep_U((U(i+1)-1)$. Let thus $i\ge 1$, and write $\rep_U(U(i)-1)=a_1\cdots a_i$ and $\rep_U(U(i+1)-1)=b_1\cdots b_{i+1}$. On the one hand, since $b_1\cdots b_i\in\NU$, we get $b_1\cdots b_i\le_{\lex}a_1\cdots a_i$. On the other hand, since $a_1\cdots a_i0\in\NU$, we get $a_1\cdots a_i0\le_{\lex}b_1\cdots b_{i+1}$, hence $a_1\cdots a_i\le_{\lex}b_1\cdots b_i$. 

Conversely, suppose that there exists an infinite word $a$ over $A_U$ such that $\rep_U(U(i)-1)=\Pref_i(a)$ for all $i\ge 0$. It is easily seen that for all $w\in A_U^*$ and all $i\in\{0,\ldots,|w|\}$, we have $\Suff_i(w)\le_{\lex}\Pref_i(a)$ if and only if $\Suff_{i+1}(w0)\le_{\lex}\Pref_{i+1}(a)$. Then we get that $U$ is Bertrand by Lemma~\ref{Lem : RepUEqualSuff}.

We now turn to the last part of the statement and we prove that $\sigma^i(a)\le_{\lex} a$ for all $i\ge1$. Suppose to the contrary that there exists $i\ge 1$ such that $\sigma^i(a)>_{\lex} a$. Then there exists $\ell\ge 1$ such that $a_i\cdots a_{i+\ell-1}>_{\lex} a_1\cdots a_\ell$, where $a=a_1a_2\cdots$. This is impossible since $a_i\cdots a_{i+\ell-1}\in\NU$ by Lemma~\ref{Lem : FacBertrand}.
\end{proof}

\begin{lemma}
\label{Lem : Bertrand}
Let $a$ be an infinite word over $\N$ such that $\sigma^i(a) \le_{\lex} a$ for all $i\ge 1$. If $a$ is not periodic, then we define $d=a$; otherwise we let $n\ge 1$ be the smallest integer such that $a=(a_1\cdots a_n)^\omega$ and we define $d=a_1 \cdots a_{n-1} (a_n+1) 0^\omega$. Then in both cases, we have $\sigma^i(d) <_{\lex} d$ for all $i\ge 1$.
\end{lemma}

\begin{proof}
The case where $a$ is not periodic is straightforward. Suppose that $a$ is periodic. 
If $i\ge n$, then $\sigma^i(d) = 0^\omega <_{\lex} d$.
For $i$ with $1\le i\le n-1$, proceed by contradiction and suppose that $\sigma^i(d) \ge_{\lex} d$, that is, $a_{i+1}\cdots a_{n-1}(a_n+1)0^\omega \ge_{\lex} a_1a_2\cdots a_{n-1}(a_n+1)0^\omega$. Then $a_{i+1}\cdots a_{n-1}(a_n+1)>_{\lex}a_1\cdots a_{n-i}$. By hypothesis on $a$, we also have $a_{i+1}\cdots a_{n-1}a_n\le_{\lex}a_1\cdots a_{n-i}$. Thus, we get $a_{i+1}\cdots a_{n-1}a_n=a_1\cdots a_{n-i}$. Moreover, by assumption on $a$, we have $\sigma^n(a)=a\ge_{\lex}\sigma^{n-i}(a)$. We then obtain that 
\[
	\sigma^i(a)	
	=a_{i+1}\cdots a_n\sigma^n(a)
	\ge_{\lex} a_1\cdots a_{n-i}\sigma^{n-i}(a)
	=a.
\] 
Since $\sigma^i(a)\le_{\lex} a$ by hypothesis, we get $\sigma^i(a)	=a$, which is impossible since $i<n$ and $n$ was chosen to be minimal for this property. 
\end{proof}

Finally, we recall the so-called Renewal theorem as stated in~\cite[Theorem 1 on p. 330]{Walters1982}; also see~\cite[Theorem 0.18]{Feller1957}.

\begin{theorem}[Renewal theorem]
\label{Thm : Renewal}
Let $(c_n)_{n\ge 0}$ and $(u_n)_{n\ge 0}$ be bounded sequences of real numbers with $0\le c_n\le 1$ and $d_n\ge 0$ for all $n$. Suppose the greatest common divisor of all integers $n$ with $c_n>0$ is $1$. Suppose $(u_n)_{n\ge 0}$ satisfies the recurrence relation $u_n=d_n+c_0u_n+c_1u_{n-1}+\cdots+c_nu_0$ for all $n$. If $\sum_{n\ge 0}c_n=1$ and $\sum_{n\ge 0}d_n<\infty$ then $\lim_{n\to\infty} u_n=(\sum_{n\ge 0}d_n)(\sum_{n\ge 0}nc_n)^{-1}$ where this is interpreted as zero if $\sum_{n\ge 0}nc_n=\infty$.
\end{theorem}

For a real number $\beta>1$, we define
\[
	S'_\beta=\{w\in \{0,\ldots,\floor{\beta}\}^\omega : \forall i\ge 0,\ \sigma^i(w)\le_{\lex} d_{\beta}(1) \}.
\]
We are now ready to show the claimed correction of Bertrand-Mathis' result.

\begin{theorem}  
\label{Thm : NewBertrand}
A positional numeration system $U$ is Bertrand if and only if one of the following occurs.
\begin{enumerate}
\item For all $i\ge 0$, $U(i)=i+1$.
\item There exists a real number $\beta>1$ such that
$\NU=\Fac(S_\beta)$. 
\item There exists a real number $\beta>1$ such that
$\NU=\Fac(S'_\beta).$
\end{enumerate}
Moreover, in Case $2$ (resp.\ Case $3$), the following hold:
\begin{enumerate}
\item[a.] There is a unique such $\beta$.
\item[b.] The alphabet $A_U$ equals $\{0,\ldots,\ceil{\beta}-1\}$ (resp.\ $\{0,\ldots,\floor{\beta}\}$).
\item[c.] We have
\begin{equation}
\label{Eq : NewBertrand}
U(i)=a_1 U(i-1)+a_2 U(i-2)+ \cdots + a_iU(0)+1
\end{equation}
for all $i\ge 0$ and 
\begin{equation}
\label{Eq : Limit}
	\lim_{i\to\infty}\frac{U(i)}{\beta^i}
	=\frac{\beta}{(\beta-1)\sum_{i=1}^\infty ia_i\beta^{-i}}	
\end{equation}
where $(a_i)_{i\ge 1}$ is $d_{\beta}^*(1)$ (resp.\ $d_{\beta}(1)$).
\item[d.] The system $U$ has the dominant root $\beta$, i.e., $\lim_{i\to\infty}\frac{U(i+1)}{U(i)}=\beta$. 
\end{enumerate} 
\end{theorem}

\begin{proof}
Let $U$ be a positional numeration system. We start with the backward direction. If $U(i)=i+1$ for all $i\ge 0$, then $\NU=0^*10^*$, hence $U$ is Bertrand. Otherwise, for the sake of clarity, write $S=\{w\in \N^\omega : \forall i\ge 0,\ \sigma^i(w)\le_{\lex} a \}$ with $a=d_\beta(1)$ or $a=d_\beta^*(1)$ as in the statement.  Suppose that $\NU=\Fac(S)$. We show that $U$ is Bertrand. Consider $y\in \NU$. There exist words $x\in \N^*$ and $z\in \N^\omega$ such that $xyz\in S$. Since $\sigma^i(xy0^\omega)\le_{\lex}\sigma^i(xyz)$ for all $i\ge 0$, we get that $xy0^\omega \in S$. Therefore $y0\in \NU$. The converse is immediate since if $y0 \in \Fac(S)$ then $y\in \Fac(S)$ as well.

Conversely, suppose that $U$ is Bertrand. By Lemma~\ref{Lem : ConvergenceBertrand}, there exists $a=a_1a_2\cdots$ such that $\rep_U(U(i)-1)=\Pref_i(a)$ and $\sigma^i(a)\le_{\lex} a$ for all $i\ge 0$. In particular, we have $a_1\ge1$ and $a_i\le a_1$ for all $i\ge 1$. If $a=10^\omega$ then $U(i)=i+1$ for all $i\ge 0$. Otherwise, let us define a new sequence $d$ from $a$. If $a$ is not (purely) periodic, define $d=a$. If $a$ is periodic and $n$ is the smallest positive integer such that $a=(a_1\cdots a_n)^\omega$, we set $d=a_1\cdots a_{n-1}(a_n+1)0^\omega$. By Lemma~\ref{Lem : Bertrand}, in both cases, we get $\sigma^i(d) <_{\lex} d$ for all $i\ge 1$. We also get $d_i\le d_1$ for all $i\ge 1$. Moreover, we have $d_1\ge1$ and $d\ne 10^\omega$ (for otherwise $a$ would also be equal to $10^\omega$). Then there exists a unique $\beta>1$ such that $d=d_\beta(1)$; see~\cite{Parry1960} or~\cite[Corollary 7.2.10]{Lothaire2002}. Also, we know that $d_\beta^*(1)=(t_1 \cdots t_{n-1} (t_n-1))^\omega$ whenever $d_\beta(1)=t_1\cdots t_n 0^\omega$ with $n\ge 1$ and $t_n \ne 0$, and that $d_\beta^*(1)=d_\beta(1)$ otherwise; again, see~\cite{Lothaire2002,Parry1960}. We get that either $a=d_\beta(1)$ or $a=d_\beta^*(1)$ depending on the periodicity of $a$. Let us show that $\NU=\Fac(\{ w \in\N^\omega : \forall i\ge 0,\ \sigma^i(w) \le_{\lex} a\}$. Consider $y\in\NU$. By Lemma~\ref{Lem : RepUEqualSuff}, we have $\Suff_i(y)\le_{\lex}\Pref_i(a)$ for all $i\le |y|$. Therefore, $\sigma^i(y0^\omega)\le_{\lex} a$ for all $i\ge 0$. Conversely, suppose that $y$ is a factor of an infinite word $w$ over $\N$ such that $\sigma^i(w) \le_{\lex} a$ for all $i\ge 0$. Then $\Suff_i(y)\le_{\lex} \Pref_i(a)$ for all $i\ge 0$. By Lemma~\ref{Lem : RepUEqualSuff}, we get $y\in\NU$.

To end the proof, we note that $A_U=\{0,\ldots,\floor{\beta}\}$ if $a=d_{\beta}(1)$ and $A_U=\{0,\ldots,\ceil{\beta}-1\}$ if $a=d_\beta^*(1)$. Moreover, since $\rep_U(U(i)-1)=a_1\cdots a_i$ for all $i\ge 0$, we get that the recurrence relation \eqref{Eq : NewBertrand} holds for all $i\ge 0$. The computation of the limit from~\eqref{Eq : Limit} then follows from Theorem~\ref{Thm : Renewal}, which in turn implies that $\lim_{i\to \infty}\frac{U(i+1)}{U(i)}=\beta$.
\end{proof}

Note that in the previous statement, the second item
coincides with the condition given in the original theorem of Bertrand-Mathis~\cite{Bertrand-Mathis1989}. The main difference between these two results is that there exist two Bertrand numeration systems associated with a \emph{simple Parry number} $\beta>1$, i.e., such that $d_\beta(1)$ ends with infinitely many zeroes. 
To distinguish them, we call \emph{canonical} the Bertrand numeration system defined by~\eqref{Eq : NewBertrand} when $a=d_\beta^*(1)$, and \emph{non-canonical} that for which $a=d_\beta(1)$.
For instance, the canonical Bertrand numeration system associated with the golden ratio $(1+\sqrt{5})/2$ is the well-known Zeckendorf numeration system $U=(1,2,3,5,8,\ldots)$ defined by $(U(0),U(1))=(1,2)$ and $U(i)=U(i-1)+U(i-2)$ for all $i\ge2$~\cite{Zeckendorf1972}. The associated non-canonical Bertrand numeration system is the numeration system $U=(1,2,4,7,12,\ldots)$ from Example~\ref{Ex : CounterExampleBertrand} defined by $(U(0),U(1))=(1,2)$ and $U(i)=U(i-1)+U(i-2)+1$ for all $i\ge2$. See Figure~\ref{Fig : NC-Zeckendorf} for automata recognizing the corresponding numeration languages. In Figures~\ref{Fig : SubCanonical3} and \ref{Fig : SubNonCanonical3}, we see the canonical and non-canonical Bertrand numeration systems associated with the integer base $3$.

\section{Linear Bertrand numeration systems} 
\label{Sec : Linear}

In the following proposition, we study the linear recurrence relations satisfied by Bertrand numeration systems associated with a \emph{Parry number} $\beta$, i.e., a real number $\beta>1$ such that $d_\beta(1)$ is ultimately periodic. As is usual, if an expansion ends with a tail of zeroes, we often omit to write it down.




\begin{proposition}
\label{Pro : BertrandParry}
Let $U$ be a Bertrand numeration system.
\begin{enumerate}
\item If $\NU=\Fac(S_\beta)$ where $\beta>1$ is such that $d_\beta^*(1)=d_1\cdots d_m(d_{m+1}\cdots d_{m+n})^\omega$ with $m\ge 0$ and $n\ge 1$, then $U$ satisfies the linear recurrence relation of characteristic polynomial 
$(X^{m+n}-\sum_{j=1}^{m+n}d_jX^{m+n-j})
	-(X^m-\sum_{j=1}^m d_jX^{m-j})$.
\item If $\NU=\Fac(S'_\beta)$ where $\beta>1$ is such that $d_\beta(1)=t_1\cdots t_n$ with $n\ge 1$ and $t_n\ge 1$, then $U$ satisfies the linear recurrence relation of characteristic polynomial 
$(X^{n+1}-\sum_{j=1}^n t_jX^{n+1-j})
	-(X^n-\sum_{j=1}^n t_jX^{n-j})$.
\end{enumerate}
\end{proposition}

\begin{proof}
Let us prove the first item. Thus, we suppose that $\NU=\Fac(S_\beta)$ where $\beta>1$ is such that $d_\beta^*(1)=d_1\cdots d_m(d_{m+1}\cdots d_{m+n})^\omega$ with $m\ge 0$ and $n\ge 1$. By Theorem~\ref{Thm : NewBertrand}, we get that 
\begin{align*}
	U(i)-U(i-n)
	&=\sum_{j=1}^i d_jU(i-j)+1
	-\sum_{j=1}^{i-n} d_jU(i-n-j)-1\\
	&=\sum_{j=1}^{m+n} d_jU(i-j)
	-\sum_{j=1}^m d_jU(i-n-j)
\end{align*}
for all $i\ge m+n$.
We now prove the second item. Suppose that $\NU=\Fac(S'_\beta)$ where $\beta>1$ is such that $d_\beta(1)=t_1\cdots t_n$ with $n\ge 1$ and $t_n\ge 1$. By Theorem~\ref{Thm : NewBertrand}, we get that 
\begin{align*}
	U(i)-U(i-1)
	&=\sum_{j=1}^n t_jU(i-j)+1
	-\sum_{j=1}^n t_jU(i-1-j)-1\\
	&=\sum_{j=1}^n t_jU(i-j)
	-\sum_{j=1}^n t_jU(i-1-j)
\end{align*}
for all $i\ge n+1$.
\end{proof}


In the following corollary, we emphasize the simple form of the characteristic polynomial in the first item of Proposition~\ref{Pro : BertrandParry} when $\beta$ is simple Parry number: the coefficients can be obtained directly from the digits of $d_\beta(1)$.

\begin{corollary}
\label{Cor : SimpleParryPolynomial}
Let $U$ be a Bertrand numeration system such that $\NU=\Fac(S_\beta)$ where $\beta>1$ is such that $d_\beta(1)=t_1\cdots t_n$ with $n\ge 1$. Then $U$ satisfies the linear recurrence relation of characteristic polynomial 
$X^n-\sum_{j=1}^n t_jX^{n-j}$. 
\end{corollary}

\begin{proof}
Since $d_\beta^*(1)=(t_1\cdots t_{n-1}(t_n-1))^\omega$, the first item of Proposition~\ref{Pro : BertrandParry} gives us that $U$ satisfies the linear recurrence relation of characteristic polynomial $
	X^n-\sum_{j=1}^{n-1}t_jX^{n-j}-(t_n-1)-1
	=X^n-\sum_{j=1}^n t_jX^{n-j}$.
\end{proof}

\section{Lexicographically greastest words of each length}
\label{Sec : LexMax}

A key argument in the proof of Theorem~\ref{Thm : NewBertrand} was the study of the lexicographically greatest words of each length;  we see this in Lemmas~\ref{Lem : RepUEqualSuff} and~\ref{Lem : ConvergenceBertrand}. In this section, we investigate more properties of these words, which will allow us to obtain yet another characterization of Bertrand numeration systems. 

In order to study the regularity of the numeration language of positional systems having a dominant root, Hollander proved the following result.

\begin{proposition}[\cite{Hollander1998}]
\label{Pro : ConvergenceRealBases}
Let $U$ be a positional numeration system having a dominant root $\beta>1$. If $\beta$ is not a simple Parry number, then $\lim_{i\to \infty} \rep_U(U(i)-1)=d_\beta(1)$. Otherwise, $d_\beta(1)=t_1\cdots t_n$ with $t_n\ne 0$ and for all $k\ge 0$, define
\[
	w_k=(t_1\cdots t_{n-1} (t_n-1))^k t_1\cdots t_n.
\]
Then for all $\ell\ge 0$, there exists $I\ge 0$ such that for all $i\ge I$, there exists $k\ge 0$ such that $\Pref_\ell(\rep_U(U(i)-1))=\Pref_\ell(w_k0^\omega)$.
\end{proposition}


\begin{example}\label{Ex : conv}\
\begin{itemize}
\item  For the integer base-$b$ numeration system $U=(b^i)_{i\ge 0}$, 
we have $w_k=(b-1)^kb$ for all $k\ge 0$ and $\rep_U(b^i-1)=(b-1)^i$ for all $i\ge 0$. This agrees with Proposition~\ref{Pro : ConvergenceRealBases}.
\item For 
the Zeckendorf numeration system,
it can be easily seen that 
\[
	\rep_U(U(i)-1)=
	\begin{cases}
	(10)^{\frac{i}{2}},	& \text{if } i \text{ is even};\\
	(10)^{\frac{i-1}{2}}1,	& \text{if } i \text{ is odd}.
	\end{cases}
\] 
We have $w_k=(10)^k 11$ for all $k\ge 0$. Therefore, for all $\ell\ge 0$ and all $i\ge \ell$, the words $\rep_U(U(i)-1)$ and $w_{\floor{\ell/2}}$ share the same prefix of length $\ell$.
\item Let $U$ be the numeration system defined by $(U(0),U(1))=(1,3)$ and for $i\ge 2$, $U(i)=3U(i-1)-U(i-2)$. Then $U$ has the dominant root $\varphi^2$ and $\rep_U(U(i)-1)=21^{i-1}$ for all $i\ge 0$. This agrees with Proposition~\ref{Pro : ConvergenceRealBases} since $d_{\varphi^2}(1)=21^\omega$. 
\end{itemize}
\end{example}

As illustrated in the next example, when $\beta$ is a simple Parry number, Proposition~\ref{Pro : ConvergenceRealBases} does not necessarily give a convergence of the sequence $(\rep_U(U(i)-1))_{i\ge 0}$. 


\begin{example}
\label{Ex : NonConvergenceRealBase}
Consider the numeration system $U=(U(i))_{i\ge 0}$ defined by $(U(0),\linebreak U(1), U(2), U(3))=(1,2,3,5)$ and for all $i\ge 4$,	$U(i) = U(i-1) + U(i-3) + U(i-4) + 1$.
It has the golden ratio as dominant root.
Hence, as for the Zeckendorf numeration system, we have $w_k=(10)^k 11$ for all $k\ge 0$.
For all $i\ge 4$, we can compute
\[
	\rep_U(U(i)-1)
	=\begin{cases}
	110^{i-2}, & \text{if } i \equiv 0,1 \bmod{4}; \\
	10110^{i-4}, & \text{if } i \equiv 2,3 \bmod{4}.
	\end{cases}
\]
Therefore, for all $i\ge 4$, $\rep_U(U(i)-1)=\Pref_i(w_0 0^\omega)$ if $i$ is congruent to $0$ or $1$ modulo $4$, and $\rep_U(U(i)-1)=\Pref_i(w_1 0^\omega)$ otherwise. Thus, the limit $\lim_{i\to \infty}\rep_U(U(i)-1)$ does not exist.
\end{example}

In Examples~\ref{Ex : conv}  and~\ref{Ex : NonConvergenceRealBase}, we illustrated that the sequence $(\rep_U(U(i)-1))_{i\ge 0}$ may or may not converge. In the first, we gave examples such that its limit is either $d_\beta(1)$ or $d^*_\beta(1)$.
In the second, we illustrated that even if the recurrence relation satisfied by $U$ gives the intuition that the sequence would converge to $w_10^\omega$, it is not the case. In fact, seeing Proposition~\ref{Pro : ConvergenceRealBases}, one might think that we can provide a positional numeration system $U$ such that $\lim_{i\to \infty} \rep_U(U(i)-1)=w_k 0^\omega$ with $k\ge 1$. We show that this cannot happen, which can be thought as a refinement of Proposition~\ref{Pro : ConvergenceRealBases}.

\begin{proposition}
\label{Pro : ConvergenceWjNotPossible}
Let $U$ be a positional numeration system with a dominant root $\beta>1$. If the limit $\lim_{i\to \infty}\rep_U(U(i)-1)$ exists, then it equals either $d_\beta^*(1)$ or $d_\beta(1)$.
\end{proposition}

\begin{proof}
If $d_\beta(1)$ is infinite, then the result follows from Proposition~\ref{Pro : ConvergenceRealBases}. Let us consider the case where $d_\beta(1)=t_1\cdots t_n$ with $t_n\ne 0$. Proceed by contradiction and suppose that there exists $k\ge1$ such that $\lim_{i\to \infty}\rep_U(U(i)-1)=w_k0^\omega$. For all $i$ large enough, $t_1\cdots t_{n-1} (t_n-1)$ is a prefix of $\rep_U(U(i)-1)$, hence the greedy algorithm implies that $\sum_{j=1}^n t_j U(i-j)>U(i)-1$. On the other hand, for all $i$ large enough, $t_1\cdots t_n$ is a factor occurring at position $kn+1$ in $\rep_U(U(i+kn)-1)$, hence, again from the greedy algorithm, we get $U(i)> \sum_{j=1}^n t_j U(i-j)$. By putting the inequalities altogether, we obtain a contradiction.
\end{proof}

Thanks to this result, we obtain another characterization of Bertrand numeration systems.

\begin{theorem}
\label{Thm : ConvergenceBertrand}
A positional numeration system $U$ is Bertrand if and only if one of the following conditions is satisfied.
\begin{enumerate}
\item We have $\rep_U(U(i)-1)=\Pref_i(10^\omega)$ for all $i\ge 0$.
\item There exists a real number $\beta> 1$ such that  $\rep_U(U(i)-1)=\Pref_i(d^*_\beta(1))$ for all $i\ge 0$.
\item There exists a real number $\beta> 1$ such that  $\rep_U(U(i)-1)=\Pref_i(d_\beta(1))$ for all $i\ge 0$.
\end{enumerate} 
\end{theorem}

\begin{proof}
All three conditions are sufficient by Lemma~\ref{Lem : ConvergenceBertrand}. Conversely, suppose that $U$ is a Bertrand numeration system. In Case 1 of Theorem~\ref{Thm : NewBertrand}, we have $\rep_U(U(i)-1)=\rep_U(i)=10^{i-1}$ for all $i\ge 1$. Otherwise, $U$ has a dominant root $\beta>1$ by Theorem~\ref{Thm : NewBertrand}. The result then follows from Lemma~\ref{Lem : ConvergenceBertrand} combined with Proposition~\ref{Pro : ConvergenceWjNotPossible}. 
\end{proof}

We note that the three cases of Theorem~\ref{Thm : ConvergenceBertrand} indeed match those described in Theorem~\ref{Thm : NewBertrand}.

\section{The non-canonical $\beta$-shift}
\label{Sec : NonCanonicalShift}

In view of their definitions, the sets $S_\beta$ and $S_\beta'$ are both \emph{subshifts} of $A^{\N}$, i.e., they are shift-invariant and closed w.r.t the topology induced by the prefix distance. These subshifts coincide unless $\beta$ is a simple Parry number. Therefore, in the specific case where $\beta$ is a simple Parry number, by analogy to the name \emph{$\beta$-shift} commonly used for $S_\beta$, we call the set $S_\beta'$ the \emph{non-canonical $\beta$-shift}. 
In this section, we see whether or not the classical properties of $S_\beta$ still hold for $S'_\beta$.

The following proposition is the analogue of~\cite[Theorem 7.2.13]{Lothaire2002} that characterizes sofic (canonical) $\beta$-shifts, i.e., such that $\Fac(S_\beta)$ is accepted by a finite automaton.

%

\begin{proposition}
\label{prop:S'-sofic}
A real number $\beta>1$ is a Parry number if and only if the subshift $S_\beta'$ is sofic.
\end{proposition}

\begin{proof}[Sketch]
If $\beta$ is not a simple Parry number, then $S_\beta=S_\beta'$ and the conclusion follows by~\cite[Theorem 7.2.13]{Lothaire2002}.
Suppose that $\beta$ is a simple Parry number for which $d_\beta(1)=t_1 \cdots t_n$ with $n\ge 1$ and $t_n\ne 0$. We get $d^*_\beta(1)=(t_1 \cdots t_{n-1}(t_n-1))^\omega$. An automaton recognizing $\Fac(S_\beta')$ can be constructed as a slight modification of the classical automaton recognizing $\Fac(S_\beta)$ given in~\cite[Theorem 7.2.13]{Lothaire2002}: we add a new final state $q'$, an edge from the state usually denoted $q_n$ (that is, the state reached while reading $t_1\cdots t_{n-1}$) to the new state $q'$ of label $t_n$ and a loop of label $0$ on the state $q'$.
\end{proof}

\begin{example}
The automata depicted in Figures~\ref{Fig : SubCanonical3} and~\ref{Fig : SubNonCanonical3}  accept $\Fac(S_3)$ and $\Fac(S'_3)$, and those of Figures~\ref{Fig : SubCanonicalPhi} and~\ref{Fig : SubNonCanonicalPhi} accept $\Fac(S_\varphi)$ and $\Fac(S'_\varphi)$.
\end{example}

A subshift $S\subseteq A^{\N}$ is said to be of \emph{finite type} if there exists a finite set of \emph{forbidden factors} defining words in $S$, i.e., if there exists a finite set $X\subset A^*$ such that $S=\{w\in A^{\N} : \Fac(w)\cap X=\emptyset\}$. It is said to be \emph{coded} if there exists a prefix code $Y\subset A^*$ such that $\Fac(S)=\Fac(Y^*)$.
It is well known that the $\beta$-shift $S_\beta$ is coded~\cite[Proposition 7.2.11]{Lothaire2002} for any $\beta>1$ and is of finite type whenever $\beta$ is a simple Parry number~\cite[Theorem 7.2.15]{Lothaire2002}. However, neither of these two properties is valid for the non-canonical $\beta$-shift $S_\beta'$ as shown by the following example.

\begin{example}
The non-canonical $\varphi$-shift $S'_\varphi$ is not of finite type. The minimal set of forbidden factors is given by the language $110^*1$ (this can be seen in Figure~\ref{Fig : SubNonCanonicalPhi}). Moreover, if $S'_\varphi$ were coded, then there would exist a prefix code $Y$ such that $\Fac(S'_\varphi)=\Fac(Y^*)$.
Since $11\in \Fac(S'_\varphi)$, we would have $x11y \in Y^*$ for some binary words $x,y$.
This would imply that $x11yx11y \in Y^*$, giving in turn $11yx11 \in \Fac(S'_\varphi)$, which is impossible.
\end{example}

%

The \emph{entropy} of a subshift $S$ of $A^\N$ can be  defined as the limit of the sequence $\frac{1}{i} \log (\mathrm{Card}(\Fac(S)\cap A^i))$ as 
$i$ tends to infinity. We refer the reader to \cite[Theorem 7.13]{Walters1982} or~\cite{Lothaire2002}. It is well known that the $\beta$-shift $S_\beta$ has entropy $\log(\beta)$. The following proposition shows that the same property holds for $S'_\beta$.

\begin{proposition}
\label{pro : SimpleParryDominantRoot}
For all real number $\beta>1$, the subshift $S'_\beta$ has entropy $\log(\beta)$.
\end{proposition}

\begin{proof}
Let $\beta>1$ be a real number. Let $U$ be the Bertrand numeration system such that $\NU=\Fac(S'_\beta)$, i.e., the numeration system defined by~\eqref{Eq : NewBertrand} with $(a_i)_{i\ge 1}=d_\beta(1)$. Since the number of length-$i$ factors of $S'_\beta$ is equal to $U(i)$, the entropy of $S_\beta'$ is given by $\lim_{i\to \infty} \frac{1}{i} \log (U(i))$. The result now follows from~\eqref{Eq : Limit}.
\end{proof}

We note that, mutatis mutandis, the same proof can be applied in order to show that the $\beta$-shift has entropy $\log(\beta)$.


Finally, whenever $\beta$ is a Parry number, we prove a relation between the number of words of each length in the canonical and the non-canonical $\beta$-shifts.

\begin{proposition}
Suppose that $\beta>1$ is a real number such that $d_\beta(1)=t_1 \cdots t_n$ with $n\ge 1$ and $t_n\ne 0$, and let $U$ and $U'$ respectively be the canonical and non-canonical Bertrand numeration systems associated with $\beta$. Then 
$U'(i+n)=U(i+n)+U'(i)$ for all $i\ge 0$.
\end{proposition}

\begin{proof}
Since $\Pref_{n-1}(d_\beta(1))=\Pref_{n-1}(d^*_\beta(1))$, we have $U'(i)=U(i)$ for all $i\in \{0,\ldots,n-1\}$. Moreover, since $t_1\cdots t_n$ is the only length-$n$ factor of $S'_\beta$ that is not present in $S_\beta$, we have $U'(n)=U(n)+1$. Hence, the statement holds for $i=0$ since $U(0)=U'(0)=1$. Now we proceed by induction. Consider $i\ge1$ and suppose that the result holds for indices less than $i$. By Theorem~\ref{Thm : NewBertrand} and Proposition~\ref{Cor : SimpleParryPolynomial}, we get that
$	U'(i+n)-U(i+n)
	=\sum_{j=1}^n t_jU'(i+n-j)+1\linebreak-\sum_{j=1}^n t_jU(i+n-j)
	=\sum_{j=1}^n t_j(U'(i+n-j)-U(i+n-j))+1$
where $U'(i+n-j)-U(i+n-j)=0$ if $j>i$, and by induction hypothesis, $U'(i+n-j)-U(i+n-j)=U'(i-j)$ if $j\le i$. As a first case, assume that $i\in\{1,\ldots,n\}$. We obtain $U'(i+n)-U(i+n)=\sum_{j=1}^i t_jU'(i-j)+1=U'(i)$ where the second equality comes from Theorem~\ref{Thm : NewBertrand}. As a second case, assume $i\ge n$. Similarly, we get $U'(i+n)-U(i+n)=\sum_{j=1}^n t_jU'(i-j)+1=U'(i)$.
\end{proof}

%
%
%
%
 \bibliographystyle{splncs04}
\bibliography{NumerationSystems.bib}

\begin{thebibliography}{10}
\providecommand{\url}[1]{\texttt{#1}}
\providecommand{\urlprefix}{URL }
\providecommand{\doi}[1]{https://doi.org/#1}

\bibitem{Bertrand-Mathis1986}
Bertrand-Mathis, A.: D\'{e}veloppement en base {$\theta$}; r\'{e}partition
  modulo un de la suite {$(x\theta^n)_{n\geq 0}$}; langages cod\'{e}s et
  {$\theta$}-shift. Bull. Soc. Math. France  \textbf{114}(3),  271--323 (1986)

\bibitem{Bertrand-Mathis1989}
Bertrand-Mathis, A.: Comment \'{e}crire les nombres entiers dans une base qui
  n'est pas enti\`ere. Acta Math. Hungar.  \textbf{54}(3-4),  237--241 (1989)

\bibitem{BruyereHansel1997}
Bruy\`ere, V., Hansel, G.: Bertrand numeration systems and recognizability.
  Theoret. Comput. Sci.  \textbf{181}(1),  17--43 (1997)

\bibitem{CharlierRampersadRigoWaxweiler2011}
Charlier, E., Rampersad, N., Rigo, M., Waxweiler, L.: The minimal automaton
  recognizing {$m\Bbb N$} in a linear numeration system. Integers
  \textbf{11B},  Paper No. A4, 24 (2011)

\bibitem{Dajani&Kraaikamp2002}
Dajani, K., Kraaikamp, C.: Ergodic theory of numbers, Carus Mathematical
  Monographs, vol.~29. Mathematical Association of America, Washington, DC
  (2002)

\bibitem{Feller1957}
Feller, W.: An introduction to probability theory and its applications. {V}ol.
  {I}. John Wiley and Sons, Inc., New York; Chapman and Hall, Ltd., London
  (1957), 2nd ed

\bibitem{FrougnySolomyak1996}
Frougny, C., Solomyak, B.: On representation of integers in linear numeration
  systems. In: Ergodic theory of {$\mathbb{Z}^d$} actions ({W}arwick,
  1993--1994), London Math. Soc. Lecture Note Ser., vol.~228, pp. 345--368.
  Cambridge Univ. Press, Cambridge (1996)

\bibitem{Hollander1998}
Hollander, M.: Greedy numeration systems and regularity. Theory Comput. Syst.
  \textbf{31}(2),  111--133 (1998)

\bibitem{Loraud1995}
Loraud, N.: {$\beta$}-shift, syst\`emes de num\'{e}ration et automates. J.
  Th\'{e}or. Nombres Bordeaux  \textbf{7}(2),  473--498 (1995)

\bibitem{Lothaire2002}
Lothaire, M.: Algebraic combinatorics on words, Encyclopedia of Mathematics and
  its Applications, vol.~90. Cambridge University Press, Cambridge (2002)

\bibitem{MassuirPeltomakiRigo2019}
Massuir, A., Peltom\"{a}ki, J., Rigo, M.: Automatic sequences based on {P}arry
  or {B}ertrand numeration systems. Adv. in Appl. Math.  \textbf{108},  11--30
  (2019)

\bibitem{Parry1960}
Parry, W.: On the {$\beta $}-expansions of real numbers. Acta Math. Acad. Sci.
  Hungar.  \textbf{11},  401--416 (1960)

\bibitem{Point2000}
Point, F.: On decidable extensions of {P}resburger arithmetic: from {A}.
  {B}ertrand numeration systems to {P}isot numbers. J. Symbolic Logic
  \textbf{65}(3),  1347--1374 (2000)

\bibitem{Renyi:1957}
R\'{e}nyi, A.: Representations for real numbers and their ergodic properties.
  Acta Math. Acad. Sci. Hungar.  \textbf{8},  477--493 (1957)

\bibitem{Schmidt1980}
Schmidt, K.: On periodic expansions of {P}isot numbers and {S}alem numbers.
  Bull. London Math. Soc.  \textbf{12}(4),  269--278 (1980)

\bibitem{Shallit1994}
Shallit, J.: Numeration systems, linear recurrences, and regular sets. Inform.
  and Comput.  \textbf{113}(2),  331--347 (1994)

\bibitem{Walters1982}
Walters, P.: An introduction to ergodic theory, Graduate Texts in Mathematics,
  vol.~79. Springer-Verlag, New York-Berlin (1982)

\bibitem{Zeckendorf1972}
Zeckendorf, E.: Repr\'esentation des nombres naturels par une somme des nombres
  de {F}ibonacci ou de nombres de {L}ucas. Bull. Soc. Roy. Sci. Li\`ege
  \textbf{41},  179--182 (1972)

\end{thebibliography}

\end{document}